\documentclass[11pt,british]{article}

\usepackage{amsfonts}
\usepackage{hyperref}
\usepackage{amssymb}
\usepackage{amsmath}
\usepackage{amsthm}
\usepackage[english=british]{csquotes}
\usepackage{stmaryrd} 
\usepackage[T1]{fontenc}
\usepackage{graphicx}

\usepackage{authblk}
\usepackage[giveninits=true,url=false,maxbibnames=5]{biblatex}
\newtheorem{theorem}{Theorem}[section]
\newtheorem{definition}[theorem]{Definition}

\newtheorem{corollary}[theorem]{Corollary}
\newtheorem*{corollaryNon}{Corollary}
\newtheorem{remark}[theorem]{Remark}

\newtheorem{observation}[theorem]{Observation}

\newcommand*{\Nat}{\mathbf N}
\DeclareMathOperator{\tower}{twr}
\newcommand*{\powerset}[1]{\mathcal P(#1)}
\newcommand*{\erarrow}{\rightarrow} 
\newcommand*{\nerarrow}{\nrightarrow} 
\newcommand*{\disj}{\mathbin{\dot\cup}} 
\newcommand*{\ordinalplus}{\mathbin{\dot{+}}} 

\newcommand*{\iinterval}[1]{\llbracket #1 \rrbracket} 
\newcommand*{\subsets}[2]{[#1]^{#2}} 
\newcommand*{\splitting}[2]{(#1\mathbin{\mid} #2)}
\newcommand*{\splitsets}[2]{\splitting{\{#1\}}{\{#2\}}}
\newcommand*{\caterpillars}{\mathcal C} 
\newcommand*{\image}[2]{#1 \mathbin{\texttt{\upshape\textquotedbl}} #2} 

\title{On Multicolour Ramsey Numbers and Subset-Colouring of Hypergraphs}

\author[1]{Bruno Jartoux}
\author[2]{Chaya Keller} 
\author[3]{Shakhar Smorodinsky}
\author[4]{Yelena Yuditsky}

\affil[1]{Department of Computer Science, Ben-Gurion University of the Negev, Be'er-Sheva, Israel 

\href{mailto:jartoux@post.bgu.ac.il}{jartoux@post.bgu.ac.il}}
           
\affil[2]{Department of Computer Science, Ariel University, Ariel, Israel

\href{mailto:chayak@ariel.ac.il}{chayak@ariel.ac.il}}
        
\affil[3]{Department of Mathematics, Ben-Gurion University of the Negev, Be'er-Sheva, Israel 

\href{mailto:shakhar@math.bgu.ac.il}{shakhar@math.bgu.ac.il}}
   
\affil[4]{Département de Mathématique, Université libre de Bruxelles, Brussels, Belgium 

\href{mailto:yuditskyL@gmail.com}{yuditskyL@gmail.com}}

\begin{document}
\maketitle
\thanks{Bruno Jartoux: Research supported by the European Research Council (ERC) under the European Union’s Horizon 2020 research and innovation programme (Grant agreement No. 678765) and by Grant 1065/20 from the Israel Science Foundation.\\Chaya Keller: Research supported by Grant 1065/20 from the Israel Science Foundation.\\Shakhar Smorodinsky: Research supported by Grant 1065/20 from the Israel Science Foundation. \\ Yelena Yuditsky: Research supported by Belgian National Fund for Scientific Research (FNRS), through PDR grant BD-OCP.}

\begin{abstract}
For $n\geq s> r\geq 1$ and $k\geq 2$, write $n \erarrow (s)_{k}^r$ if every hyperedge colouring with $k$ colours of the complete $r$-uniform hypergraph on $n$ vertices has a monochromatic subset of size $s$. Improving upon previous results by \textcite{AGLM14} and \textcite{EHMR84} we show that 
\[
\text{if } r \geq 3 \text{ and } n \nerarrow (s)_k^r \text{ then } 2^n \nerarrow (s+1)_{k+3}^{r+1}.
\]
This yields an improvement for some of the known lower bounds on multicolour hypergraph Ramsey numbers.

Given a hypergraph $H=(V,E)$, we consider the Ramsey-like problem of colouring all $r$-subsets of $V$ such that no hyperedge of size $\geq r+1$ is monochromatic. We provide upper and lower bounds on the number of colours necessary in terms of the chromatic number $\chi(H)$. In particular we show that this number is $O(\log^{(r-1)} (r \chi(H)) + r)$.
\end{abstract}

\section{Introduction}

Even though Ramsey theory has attracted much attention from its inception almost a century ago, many questions remain elusive. This paper is primarily concerned with lower bounds for (multicolour, hypergraph) Ramsey numbers and the methods that yield them.

\paragraph{Notations.}
For any natural number $n\in\Nat$, put $\iinterval{n}=\{0,\dots,n-1\}$. For any set $S$ and $k\in\Nat$, the set $\subsets S k$ is $\{T\subset S \colon |T|=k\}$, i.e., the set of all $k$-subsets of $S$. We write $\powerset S$ for the powerset of $S$.
If $f\colon A \to B$ is a function and $X\subset A$, we write $\image{f}{X}=\{f(x)\colon x \in X\}$ instead of the more usual (outside set theory) but ambiguous $f(X)$. Throughout the paper, $\log$ is the \emph{binary} logarithm (although the choice of base is inconsequential in most places).

\begin{definition}[Rado's arrow notation]\label{def:arrows}
 Given $k\geq 2$, $n>r\geq 1$, \emph{an $r$-subset $k$-colouring of $\iinterval n$} is a function $f\colon \subsets{\iinterval n}{r} \to \iinterval k$. A set $X\subset \iinterval n$ is \emph{monochromatic (under $f$, in the colour $i\in\iinterval k$)} if $\subsets X r \subset f^{-1}(i)$, or equivalently $\image f {\subsets X r} = \{i\}$. For $k$ integers $(s_i)_{i\in\iinterval k}$ all satisfying $n\geq s_i >r$, we write
\[ n \erarrow (s_0,\dots,s_{k-1})^r, \]
or more concisely $n \erarrow (s_i)_{i\in\iinterval k}^r$, to mean that for every $f\colon \subsets{\iinterval n}{r} \to \iinterval k$ there is a colour $i\in\iinterval k$ in which a subset of $\iinterval n$ of size $s_i$ is monochromatic. When $s_0=s_1=\dots=s_{k-1}=s$ (the \enquote{diagonal case}) this is further abbreviated to $n \erarrow (s)_{k}^r$.

The logical negation of any arrow relation is written similarly, replacing $\erarrow$ with $\nerarrow$.
\end{definition}

Our main result is as follows. 
\begin{theorem}\label{thm:main}
Fix integers $k\geq 2$ and  $n>r\geq 3$ and $k$ integers $(s_i)_{i\in\iinterval k}$ all satisfying $n\geq s_i \geq r+1$.
\begin{gather*}
\text{If } n \nerarrow (s_i)_{i\in \iinterval k}^{r}
\text{ then } 
    2^n \nerarrow (s_0+1,\dots,s_{k-1}+1,\underbrace{r+2,\dots,r+2}_{\substack{\eta(r) \text{ terms,}\\\text{each $r+2$}}})^{r+1},
\end{gather*}
with the integer 
\begin{equation}
  \eta(r) = \begin{cases}
1 &\text{if } r = 3,\\
2 &\text{if } r > 3 \text{ is even,}\\
3 &\text{if } r > 3 \text{ is odd. }
\end{cases} \label{eq:additional_colours}  
\end{equation}
\end{theorem}

This result extends to transfinite cardinals, strengthening results of Erd\H{o}s et al.~\cite[Chap. 24]{ER52} in certain ranges of parameters (Appendix \ref{app:infinite}).

Let $n\geq s> r\geq 1$ and $k\geq 2$. The \emph{(multicolour, hypergraph, diagonal) Ramsey number}\footnote{Different authors use different notations: cf.\@ $r_k(K_s^r)$ \cite{AGLM14}, $r_r(s;k)$ \cite{CFS13}, $R_k(s;r)$ \cite{R17}.} $r_k(s;r)$ is the smallest $n$ for which $n\erarrow (s)_k^r$. The fact that these numbers exist is Ramsey's 1930 theorem \cite{R30}.

As a corollary to Theorem \ref{thm:main}, we obtain improved lower bounds for multicolour Ramsey numbers. The \emph{tower functions} are defined by $\tower_1(x)=x$ and $\tower_{r+1}(x)=2^{\tower_r(x)}$. 

\begin{corollary}\label{cor:Ramsey_bounds}
	There are absolute constants $\alpha\simeq 1.678$ and $\beta$ such that for $r=3$ and any $k\geq 4$ or for $r\geq 4$ and $k\geq \lfloor 5 r/2\rfloor -5$, we have: 
\begin{equation}
    r_{k}(r+1;r)> \tower_r \left(\frac \alpha  2 \cdot \left(k-\frac{5r}{2}\right) + \beta\right), \label{eq:bound_r+1}
\end{equation}	
and for $r=3$ and any $k\geq 2$ or for $r\geq 4$ and $k\geq \lfloor 5 r/2\rfloor -7$,
\begin{equation}
     r_{k}(r+2;r)> \tower_r \left(\alpha  \cdot \left(k-\frac{5r}{2}\right) + \beta\right). \label{eq:bound_r+2}
\end{equation}
\end{corollary}

\paragraph{Subset colouring in hypergraphs.}
Our second result addresses a hypergraph colouring problem. 
Given a hypergraph $H$ and $r \in \Nat$, we are interested in the minimum number $k=k(H;r)$ for which there exists a $k$-colouring of all $r$-subsets of vertices without any monochromatic hyperedge of size $\geq r+1$. Note that if all the hyperedges in $H$ are of size at least $2$, then $k(H;1)=\chi(H)$, the standard \emph{vertex chromatic number} of $H$ (i.e.\@ the least number of colours in a colouring of $V$ in which no hyperedge is monochromatic).

In a work on simplicial complexes, Sarkaria related this parameter (which he called the \emph{weak $r$-th chromatic number}), to embeddability properties \cite{Sarkaria87,Sarkaria81,Sarkaria83}.

We show (Theorem \ref{thm:equivalence}) that for any $H$, the number of colours $k(H;r)$ is not much larger than the corresponding number of colours $k(K_{r\chi(H)}^{(r+1)};r)$ where the hypergraph $K_{r\chi(H)}^{(r+1)}$ is the complete $(r+1)$-uniform on $r\chi(H)$ vertices. Note that $\chi(K_{r\chi(H)}^{(r+1)})=\chi(H)$. This resembles the Erd\H{o}s--Stone theorem \cite{ES46} (the chromatic number of a graph essentially controls its Tur\'{a}n number). Hence finding $k(n,r)= \max\{k(H;r) \colon \chi(H)=n\}$ for any $n,r$, is essentially equivalent to the problem of finding the Ramsey number $r_{k'}(r+1;r)$ for an appropriate value of $k'$.

With our lower bound on multicolour Ramsey numbers and the previously known upper bound, we obtain (Corollary \ref{cor:UBhypergraphs}):
\[\Omega\left(\frac{\log^{(r-1)}(rn)}{\log^{(r)}(rn)}\right) < k(n,r) < O\left(\log^{(r-1)}(rn) + r\right).
\]

\subsection{Previous results} 

The values of $r_k(s;r)$ remain unknown except for several simple cases; in fact, for $r\geq 3$ only $r_2(4;3)=13$ is known \cite{MR91}. For general $r$ the only known bounds are
\begin{equation}
\label{eq:RamseyBounds}
 \tower_r(c'k)  \leq r_k(s;r) \leq \tower_r(ck \log k)
\end{equation}
for sufficiently large $s$, with $c$ and $c'$ functions of $s$ and $r$ \cite{EHR65,ER52}. 
See the book by \textcite{GRS90} for background on finite Ramsey theory and the survey by \textcite{R17} for recent bounds. 

Some lower bounds on Ramsey numbers are obtained through \emph{stepping-up lemmata}. The following (negative) stepping-up lemma attributed is due Erd\H{o}s and Hajnal \cite{GRS90}. It transforms lower bounds on $r_k(\,\cdot\,;r)$ into lower bounds on $r_k(\,\cdot\,;r+1)$, as follows:
\begin{gather}
\text{if } r \geq 3 \text{ and } n \nerarrow (s)_k^r \text{ then } 2^n \nerarrow (2s+r-4)_k^{r+1}, \label{eq:erdos_hajnal}
\intertext{and, for stepping up from 2 to 3:}
\text{if } n \nerarrow (s)_k^2 \text{ then } 2^n \nerarrow (s+1)_{2k}^{3}, \label{eq:step_23_k}
\intertext{where the number of colours doubles. \textcite{CFS13} improved \eqref{eq:erdos_hajnal}:}
\text{if } r \geq 4 \text{ and } n \nerarrow (s)_k^r \text{ then } 2^n \nerarrow (s+3)_k^{r+1}.
\intertext{Furthermore for odd $r$ or for $k\geq 3$, they obtain}
\text{if } r \geq 4 \text{ and } n \nerarrow (s)_k^r \text{ then } 2^n \nerarrow (s+2)_k^{r+1}.
\intertext{Repeated application of this last lemma alone will not guarantee the lower bound of (\ref{eq:RamseyBounds}) where $s$ and $r$ are both large and $s< 2r$. For example, it cannot deal with the case where the size of the hyperedges grows by exactly $1$. This is why \textcite{AGLM14} implicitly establish the following variant:}
\text{if } r \geq 2 \text{ and } n \nerarrow (s)_k^r \text{ then } 2^n \nerarrow (s+1)_{2k+2r-4}^{r+1}, \label{eq:step_Axenovich}
\intertext{which implies, in the left inequality of (\ref{eq:RamseyBounds}), that 
for any $s>r$ and any $k >r2^r$, we have
	$r_k(s;r) > \tower_r\left(\frac k {2^r}\right)$.}
\intertext{We mention one more earlier result by \textcite[Lemma 24.1]{EHMR84}:}
\text{if } r \geq 3 \text{ and } n \nerarrow (s)_k^r \text{ then } 2^n \nerarrow (s+1)_{k+2^{r} + 2^{r-1} -4}^{r+1}. \label{eq:step_EHMR}
\end{gather}

\subsection{Organization of the paper}
Section \ref{sec:ramsey} establishes our stepping-up lemma (Theorem \ref{thm:main}). In Section \ref{sec:ramsey_bounds} we apply said lemma to improve the lower bounds for several multicolour hypergraph Ramsey numbers. Section \ref{sec:subset_colouring} deals with subset-colouring of hypergraphs. Appendix \ref{app:infinite} extends the stepping-up lemma to infinite cardinals. Appendix \ref{app:graph_bounds} explains our starting bound on $r_k(3;2)$.

\section{A stepping-up lemma}
\label{sec:ramsey}

In this section we prove Theorem \ref{thm:main}. To this end let $k$, $n$, $r$, $(s_i)_{i\in\iinterval k}$ be as in the hypotheses. In particular, there exists $f_{r}:\subsets{\iinterval n} {r} \to \iinterval k$ under which, for every $i\in\iinterval k$, no set of size $s_i$ is monochromatic in colour $i$. We will use $f_{r}$ to construct $f_{r+1}:\subsets{\iinterval {2^n}} {r+1} \to \iinterval {k+\eta(r)}$ under which, for every $i\in\iinterval k$, no set of size $s_i+1$ is monochromatic in colour $i$, and no set of size $r+2$ is monochromatic in either of the additional colours $k$, $k+1$, \dots, $k+\eta(r)-1$.

\subsection{Preliminary results and definitions}

Before proceeding to the proof of the theorem, we need several definitions and results. 

\paragraph{Splitting indices.} 
For every natural number $n$, let $d(n)\subset \Nat$ be the unique finite set of integers such that $n = \sum_{i\in d(n)} 2^i$. In other words, $d(n)$ is the set of non-zero indices in the binary representation of $n$. Given a finite set $S$ of natural numbers, $|S| \geq 2$, its \emph{first splitting index} is
\[ s(S) = \max \{ i \in \Nat \colon \exists x,y\in S \colon  i \in d(x)\setminus d(y)\}. \]
That is, $s(S)$ is the index of the most significant bit where two elements in $S$ differ in their binary representation. 

The first splitting index partitions $S$ into two disjoint, non-empty subsets $S_{0}=\{x \in S \colon s(S) \notin d(x)\}$ and $S_{1}=\{x \in S \colon s(S) \in d(x)\}$ with $\max S_0 < \min S_1$. This partition is unique and exists as soon as $|S|\geq 2$; we use the notation $\splitting {S_0} {S_1}$ to denote the set $S$ where the partition $S_0\dot\cup S_1$ is the partition obtained at the first splitting index.

\begin{observation}\label{obs:heredity}
Let $S=\splitting{S_0}{S_1}$ be a set of at least two natural numbers. If $\emptyset\neq L\subset S_0$ and $\emptyset \neq R \subset S_1$ then $s(L\cup R) = s(S)$ and $L\cup R=\splitting{L}{R}$.
\end{observation}

\paragraph{Caterpillars.}
We define a family $\caterpillars\subset \powerset\Nat$, whose elements we call \emph{caterpillars} (due to their resemblance with the caterpillar trees of graph theory), as follows: The empty set and all singletons are caterpillars. A finite set of 2 or more natural integers, with partition $\splitting {S_0} {S_1}$, is a caterpillar if and only if at least one of $S_0$ and $S_1$ is a singleton and the other is a caterpillar.

For example, the reader may check that $\{0,1,2,4\}=
\splitting{\splitting{\splitsets{0}{1}}{\{2\}}}{\{4\}}\in\caterpillars$ and $\{0,1,2,3\}= \splitsets{0,1}{2,3}\notin\caterpillars$, by considering binary representations of the integers. One can depict the sequence of splits introducing the partitions using a full binary tree (see Figure \ref{fig:caterpillars}). Caterpillars are those subsets whose binary tree possesses a dominating root-to-leaf path, that is, the vertex set of the tree can be partitioned into a path and a collection of vertices where each vertex has a neighbour on the path. 

\begin{figure}[ht]
    \centering
    \includegraphics[width=\textwidth]{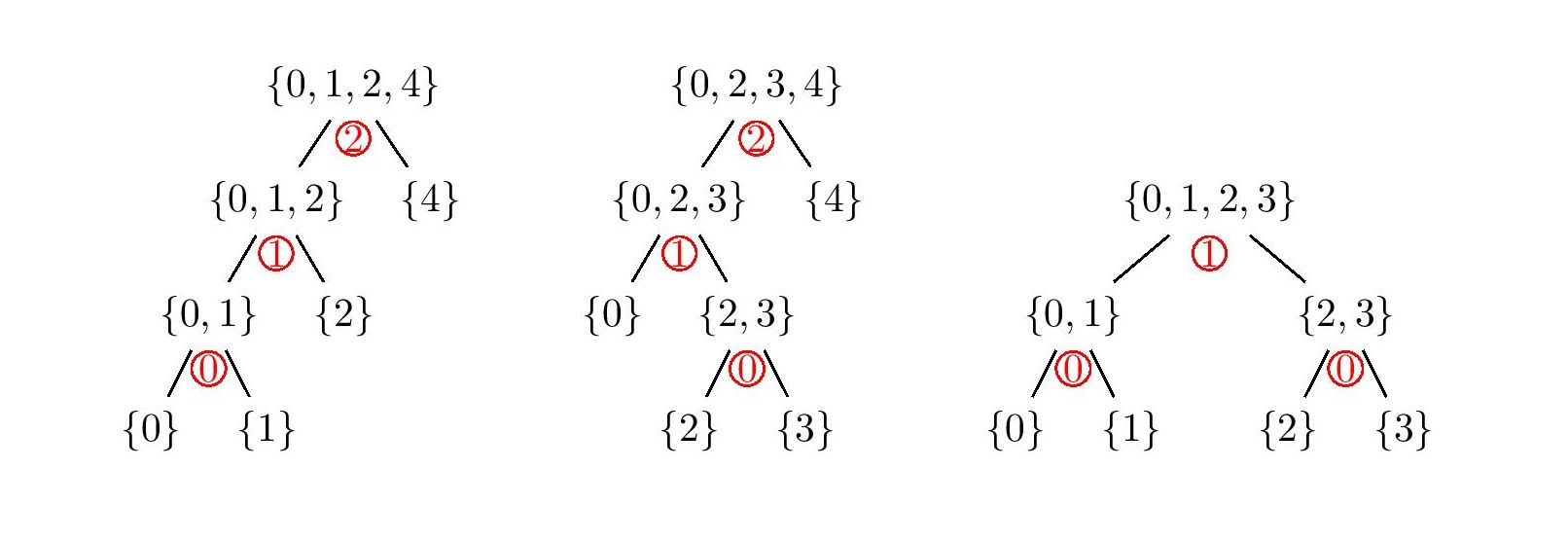}
    \caption{The tree structures of two caterpillars (left) and a non-caterpillar (right). The splitting indices are marked in red and circled.}
    \label{fig:caterpillars}
\end{figure}

The proofs by \textcite{EHMR84} or \textcite{AGLM14} use a smaller class of thin sets, namely, the caterpillars whose binary tree representations only have leaves for their proper left subtrees.

\begin{observation}
\label{obs:equivalence}
Let $S\subset \Nat$ be a finite set. The following are equivalent:
\begin{enumerate}
    \item $S\notin \caterpillars$.
    \item There is a subset $\splitting L R \subset S$ with $|L|,|R|\geq 2$.
    \item There is a subset $\splitting L R \subset S$ with $|L|=|R|=2$.
\end{enumerate}
As a simple consequence, $\caterpillars$ is transitive: if $S\in \caterpillars$ then $\powerset S\subset \caterpillars$.
\end{observation}
\begin{proof}
    The implication $(1)\implies (2)$ can be easily proved by induction on $|S|$. Indeed, if $\splitting{S_0}{S_1}\notin \caterpillars$ then either $S_0\notin \caterpillars$, or $S_1\notin \caterpillars$, or $|S_0|,|S_1|\geq 2$.
    
    The implication $(2) \implies (3)$ holds by Observation \ref{obs:heredity}.
    
    To show $(3) \implies (1)$, assume that $(3)$ holds but $(1)$ does not. By $(3)$, there are four integers $a < b < c < d$ such that $\splitting{\{a,b\}}{\{c,d\}} \subset S\in \caterpillars$. By the definition of $\caterpillars$, $S=\splitting{S_0}{S_1}$ where either $|S_0|=1$ or $|S_1|=1$, without loss of generality, $|S_0|=1$. It must be the case that $S_0\cap \{a,b,c,d\}=\emptyset$ as otherwise $|S_0|>1$. Indeed, assume without loss of generality that $a\in S_0$, then $c\notin S_0$ but because $\splitting{\{a,b\}}{\{c,d\}} \subset S$ we get that $b\in S_0$. Hence $\{a,b,c,d\}\subseteq S_1$ and by induction on the size of the set we consider we get a contradiction to $(3)$. 
\end{proof}

With each caterpillar $S\in\caterpillars$ we associate a subset of indices $\delta(S)=\{s(\{x,y\}) \colon x,y\in S, x\neq y \}$. For example, the reader may check that $\delta(\{1,3,6,31\})= \{1,2,4\}$.

\begin{observation}\label{obs:delta}
If $S\in\caterpillars\setminus \{\emptyset\}$ then $|\delta(S)|=|S|-1$.
\end{observation}

\begin{proof}
    Simple induction using Observation \ref{obs:heredity}.
\end{proof}

\begin{observation}\label{obs:subset_and_delta}
For every $S\in\caterpillars\setminus\{\emptyset\}$ and $r\in\Nat$, we have $\subsets{\delta(S)}{r}=\image \delta {\subsets{S}{r+1}}$.
\end{observation}

\begin{proof}
    Note that ${\subsets{S}{r+1}}\subset \caterpillars$, so the right-hand term is well-defined. 

    For $r=0$ or $|S|= 1$ this is clear: the equality becomes $\{\emptyset\}=\{\emptyset\}$ if $r=0$ or $\emptyset=\emptyset$ if $r>0$ and $|S|=1$. For larger $|S|$ and $r$, let $x\in S$ be such that $S= \splitting{L}{\{x\}}$, $L\in \caterpillars \setminus\{\emptyset\}$ and $s(L)<s(S)$. 
    Then $\delta(S)= \{s(S)\}\disj \delta(L)$, where the union is disjoint. Thus 
    \begin{align*}
        \subsets{\delta(S)}{r} &= \{\{s(S)\}\cup D \colon D\in \subsets{\delta(L)}{r-1}\} \disj \subsets{\delta(L)}{r}.
        \intertext{By induction on $r$ and the size of the set considered this is equal to}
        \subsets{\delta(S)}{r} &= \{\{s(S)\}\cup \delta(Y) \colon Y\in \subsets{L}{r}\} \disj \{\delta(Y)\colon Y\in\subsets{L}{r+1}\}\\
        \intertext{Thus,}
        \subsets{\delta(S)}{r} &= \{\delta(Y\disj\{x\}) \colon Y\in \subsets{L}{r}\} \disj \{\delta(Y)\colon Y\in\subsets{L}{r+1}\}\\
        &=\{\delta(Y)\colon Y \in \subsets{S}{r+1}\}=\image \delta {\subsets{S}{r+1}}.\qedhere
    \end{align*}
\end{proof}

Each finite set $S=\splitting L R$ of natural numbers that is not a caterpillar has a \emph{type} $t(S)\in\Nat\times\Nat$, defined by induction as follows: 

\[
t(\splitting{L}{R}) = 
\begin{cases}
(|L|,|R|) &\text{if $|L|,|R|\geq 2$,}\\
t(L) &\text{if $|R|=1$,}\\
t(R) &\text{if $|L|=1$.}
\end{cases}
\]

For example, $t(\{0,1,2,3,4,8\})=t(\{0,1,2,3\})=(2,2)$.

\begin{observation}\label{obs:types}
Let $S$ be a finite, nonempty set of integers.
\begin{enumerate}
    \item[(1)] If $S$ has type $(p,q)$ with $p\geq 3$ then some subset of $S$ of cardinality $|S|-1$ has type $(p-1,q)$ (respectively $q\geq 3$ and $(p,q-1)$).
    \item[(2)] If $S$ has type $(2,q)$ then some subset of $S$ of cardinality $|S|-1$ is a caterpillar or has type $(p',q')$ with $p'+q'\leq q$ (respectively $(p,2)$ and $p'+q'\leq p$).
    \item[(3)] If $S$ has type $(p,q)$ with $p+q < |S|$ then some subset of $S$ of cardinality $|S|-1$ also has type $(p,q)$.
    \item[(4)] The caterpillars are exactly the finite subsets of $\Nat$ without subsets of type $(2,2)$. 
\end{enumerate}
\end{observation}
\begin{proof}
    \begin{enumerate}
        \item[(1)] We prove by induction on $|S|$. If $|S|=p+q$ then $S=\splitting L R$ with $|L|=p$ and $|R|=q$. Let $x$ be any element of $L$, then by Observation \ref{obs:heredity} the set $S\setminus\{x\} = (L\setminus\{x\},R)$ has type $(p-1,q)$.
        
        If $|S|>p+q$, then without loss of generality $S=\splitting{L}{\{x\}}$ with $t(L)=(p,q)$. By induction, there is $L'\subset L$ with $|L'|=|L|-1=|S|-2$ and $t(L')=(p-1,q)$. 
        Then $L' \cup \{x\} \subset S$ with $|L' \cup \{x\} | = |S|-1$, and $ t (L' \cup \{x\})=t(L')=(p-1,q)$. 
        \item[(2)] By the same induction argument as above, it suffices to consider the case $|S|=2+q$. Then  $S=\splitting {\{x,y\}} R$ with $x\neq y$ and $|R|=q$. Hence by Observation~\ref{obs:heredity}, $S \setminus \{x\} = \splitting {\{y\}} R$. If $R$ is a caterpillar, then $S \setminus \{x\}$ is a caterpillar, otherwise $S \setminus \{x\}$ has type $(p',q')$ with $p'+q' \leq |R|=q$.
        \item[(3)] In this case, w.l.o.g., $S=\splitting {\{x\}} R$ with $R$ of type $(p,q)$. Then $S \setminus \{x\}$ is of type $(p,q)$.
        \item[(4)] Direct reformulation of Observation \ref{obs:equivalence} (1),(3).\qedhere
    \end{enumerate}
\end{proof}

\subsection{Description and validity of the colouring.} 

\paragraph{Description.} 

We can now define our colouring $f_{r+1}\colon \subsets{\iinterval{2^n}} {r+1} \to \iinterval {k+3}$. 
Let $S\in \subsets{\iinterval{2^n}} {r+1}$.

If $S\in\caterpillars$ then $\delta(S)\in \subsets{\iinterval n}{r}$, and we let $f_{r+1}(S)=f_{r}(\delta(S)) \in \iinterval k$. Otherwise $S\notin\caterpillars$ has a type $t(S)=(p,q)$  with $p,q\geq 2$ and $p+q\leq r+1$, and we let
\[
f_{r+1}(S)=\begin{cases}
    k &\text{ if } p+q=r+1,\ p \text{ even}, \\
    k+1 &\text{ if } p+q<r+1,\ p+q \text{ even},\\
    0 &\text{ if } p+q=r+1,\ p \text{ odd and $r$ odd},\\
    k+2 &\text{ if } p+q=r+1,\ p \text{ odd and $r$ even},\\
    1 &\text{ if } p+q<r+1,\ p+q \text{ odd.}
\end{cases}
\]

\paragraph{Validity.} We detail the proof that this colouring has the expected properties for $r$ odd. The case of even $r$ is similar and easier. 
Let $X\subset \iinterval{2^n}$ be monochromatic in colour $i$ under $f_{r+1}$. We will show that $|X| \leq s_i$ (if $i\in \iinterval k$) or $|X|\leq r+1$ (for $i\in\{k,k+1,k+2\}$). As $r+1\leq s_i$ we may already assume $|X|> r+1$.

First consider the case $X\in \caterpillars$. Then by Observation~\ref{obs:equivalence}, $\subsets X {r+1} \subset \caterpillars$, and thus \[\{i\} = \image{f_{r+1}}{\subsets X {r+1}} =\image{f_{r}}{(\image\delta {\subsets X {r+1}})} \subset \iinterval k,\] so we have to show $|X|\leq s_i$. By Observation \ref{obs:subset_and_delta}, the set $\delta(X)\subset \iinterval n$ is monochromatic in colour $i$ under $f_{r}$, as
\[ \image {f_{r}}{\subsets{\delta(X)}{r}} = \image {f_{r}}{(\image \delta {\subsets{X}{r+1}})}  = \{ i\}.
\]
This implies by the hypothesis on $f_{r}$ and Observation~\ref{obs:delta} that $|X|-1 = |\delta(X)| < s_i$, as required.

Now consider the case $X\notin \caterpillars$. We can assume $|X|=r+2$. The set $X$ has a type $(P,Q)$ with $P,Q\geq 2$ and $P+Q\leq r+2$. We consider several cases covering all possible values of $P$ and $Q$, and we argue that in each case two $(r+1)$-subsets of $X$ receive distinct colours under $f_{r+1}$, i.e. $|\image{f_{r+1}}{\subsets X {r+1}}|\geq 2$.

\begin{itemize}
    \item If $P+Q = r+2$ and $P,Q\geq 3$, then by Observation~\ref{obs:types} (1), $X$ contains $(r+1)$-subsets $S,S'$ of types $(P-1,Q)$ and $(P,Q-1)$, respectively, so also the pair of colours $\{0,k\}=\{f_{r+1}(S),f_{r+1}(S')\}$.
    \item If $(P,Q) = (2,r)$: on one hand, by Observation~\ref{obs:types} (1), there is an $(r+1)$-subset $S$ with the type $(2,r-1)$ and $f_{r+1}(S)=k$. On the other hand, by Observation~\ref{obs:types} (2), there is an $(r+1)$-subset $S'$ which is either a caterpillar, so $f_{r+1}(S')\in \iinterval k$, or of a type $(x,y)$ with $x+y\leq r$, so $f_{r+1}(S')\in \{1,k+1\}$.
    \item If $(P,Q) = (r,2)$: symmetric to the previous case, because $r-1$ is even.
    \item If $P+Q\leq r+1$ and $P \geq 3$, then by Observation~\ref{obs:types} (3), $X$ contains an $(r+1)$-subset $S$ of type $(P,Q)$ and by Observation~\ref{obs:types} (1) an $(r+1)$-subset $S'$ of type $(P-1,Q)$.
    Since $S'$ is not a caterpillar and $P-1+Q<r+1$, $f_{r+1}(S')\in \{1,k+1\}$.
    If $f_{r+1}(S) \in \{0,k\}$ then we are done. Otherwise, $f_{r+1}(S) \in \{1,k+1\}$, and the distinct parities of $P+Q$ and $P-1+Q$ give $\{f_{r+1}(S),f_{r+1}(S')\}= \{1,k+1\}$
    \item If $P+Q\leq r+1$ with $Q \geq 3$, same argument as the previous case.
    \item If $(P,Q) = (2,2)$: because $r+2\geq 5$, by Observation~\ref{obs:types} (3), there is an $(r+1)$-subset $S$ of type $(2,2)$, whose colour is either $k$ (if $r=3$) or $k+1$. On the other hand, by Observation \ref{obs:types} (2), there is an $(r+1)$-subset $S'$ which is a caterpillar, and so is coloured with some colour between 0 and $k-1$.
\end{itemize}

\begin{observation}
For odd $r$ the function $f_{r+1}$ does not use the colour $k+2$, and in the case of $r=3$ the function $f_4$ also does not use the colour $k+1$.
\end{observation}

This completes the proof of Theorem \ref{thm:main}. We conclude with an observation which will be useful in Section \ref{sec:ramsey_bounds}.

\begin{observation}
\label{obs:4caterpillar}
If $n \nerarrow (3)_k^2$ then $2^n \nerarrow (5)_k^3$.
\end{observation}

\begin{proof}
    First observe that every set of five natural numbers includes a caterpillar of size four. Let the elements of $S\in \subsets \Nat 5$ be $a,b,c,d,e$. If $S=\splitsets{a}{b,c,d,e}$ then $\{a,b,c,d\}\in\caterpillars$. (Note that any 3-subset of $\Nat$ is a caterpillar.) If $S=\splitsets{a,b}{c,d,e}$ then $\{a,c,d,e\}\in\caterpillars$. Other cases are symmetric.
    
    As in the proof of Theorem \ref{thm:main}, the colours in $\iinterval k$ suffice for colouring $[\iinterval{2^n}]^3$ such that any caterpillar $S \subset \iinterval{2^n}$ with $|S|=4$ is not monochromatic. Any 5-subset of $\iinterval{2^n}$ contains such an $S$, and thus is itself not monochromatic under the same colouring.
\end{proof}

\section{Lower bounds on multicolour hypergraph Ramsey numbers}
\label{sec:ramsey_bounds}

Erd\H{o}s, Hajnal and Rado obtained the first bounds on $r_k(s;r)$ \cite{EHR65,ER52}. Note the following:
\begin{observation}\label{obs:monotony}
The quantity $r_k(s;r)$ is nondecreasing in $s$ and in $k$. The quantity $r_k(r+1;r)$ is nondecreasing in $r$.
\end{observation}

\begin{theorem}[\citeauthor{ER52}\cite{ER52}, \citeauthor{EHR65}\cite{EHR65}]
	\label{thm:ER}
	Let $r \geq 2$. There exists $s_0(r)$ such that:
	\begin{itemize}
		\item For any $s> s_0$, we have $r_k(s;r) \geq \tower_r(c'k)$, 
		\item For any $s>r$, we have $r_k(s;r) \leq \tower_r(ck \log k)$, 
	\end{itemize}
	where $c'=c'(s,r)$ and $c=c(s,r) \leq 3(s-r)$.
\end{theorem}  
Unlike the upper bound, the lower bound holds only for sufficiently large values of $s$. Duffus, Leffman and R\"{o}dl~\cite{DLR95} gave a tower-function lower bound for all $s \geq r+1$, but said bound, $\tower_{r-1}(c''k)$, is significantly weaker than the bound of Theorem~\ref{thm:ER}. Conlon, Fox, and Sudakov~\cite{CFS13} proved that the lower bound of Theorem~\ref{thm:ER} holds whenever $s \geq 3r$. Axenovich et al.~\cite{AGLM14} matched the lower bound of Theorem~\ref{thm:ER} for all $s>r$, but only for sufficiently large $k$. 
\begin{theorem}[\citeauthor{AGLM14}]
	\label{thm:AGLM}
	For any $s>r \geq 2$ and any $k >r2^r$, we have
	\[
	r_k(s;r) > \tower_r\left(\frac k {2^r}\right).
	\]
\end{theorem} 
In particular the results of~\cite{CFS13} give $r_k(s;3) \geq 2^{2^{c'k}}$ for some constant $c'$ whenever $s \geq 9$, and the results of~\cite{AGLM14} give $r_k(4;3) > 2^{2^{k/8}}$ for all $k>24$.

We prove that the lower bound of Theorem~\ref{thm:ER} holds for values of $k$ much closer to $r$ than in Theorem \ref{thm:AGLM}, and also improve the constant inside the tower function. 

\begin{corollaryNon}[\ref{cor:Ramsey_bounds}]
	There are absolute constants $\alpha\simeq 1.678$ and $\beta$ such that for $r=3$ and any $k\geq 4$ or for $r\geq 4$ and $k\geq \lfloor 5 r/2\rfloor -5$, we have: 
\begin{equation}
    r_{k}(r+1;r)> \tower_r \left(\frac \alpha  2 \cdot \left(k-\frac{5r}{2}\right) + \beta\right), 
\end{equation}	
and for $r=3$ and any $k\geq 2$ or for $r\geq 4$ and $k\geq \lfloor 5 r/2\rfloor -7$,
\begin{equation}
     r_{k}(r+2;r)> \tower_r \left(\alpha  \cdot \left(k-\frac{5r}{2}\right) + \beta\right). 
\end{equation}
\end{corollaryNon} 

\begin{proof}
Fix some integers $K\geq 0$ and $n,s\geq 4$ such that $n \nerarrow (s)_{K+2}^3$. Through either $2t+1$ or $2t+2$ repeated applications of Theorem \ref{thm:main} we get, for any $t\in \Nat$,
\begin{align}
    \tower_{2t+2} n &\nerarrow (s+2t+1)_{K+5t+3}^{4+2t} \label{eq:tower_2t+2}\\
   \text{and } \tower_{2t+3} n &\nerarrow (s+2t+2)_{K+5t+5}^{5+2t}. \label{eq:tower_2t+3}
\end{align}

Equivalently, in terms of Ramsey numbers,
\begin{align}
    r_{K+5t+3}(s+2t+1;4+2t)&> \tower_{2t+2} n \label{eq:ramsey_2t+2}\\
   \text{and } r_{K+5t+5}(s+2t+2;5+2t)& >  \tower_{2t+3} n \label{eq:ramsey_2t+3}.
\end{align}

It is known (see Appendix \ref{app:graph_bounds}) that there are absolute constants $C_0>0$ and $C_1= 1073^{1/6}\simeq 3.199$ such that, for every integer $L\geq 0$,
\[\lceil C_0\cdot C_1^{L+2} \rceil \nerarrow (3)_{L+2}^2.\]
 From there we have both $2^{\lceil C_0\cdot C_1^{L+2}\rceil} \nerarrow (4)_{2L+4}^3$ and $2^{\lceil C_0\cdot C_1^{L+2}\rceil}\nerarrow (5)_{L+2}^3$, which establishes the Theorem for $r=3$. The first relation follows from \eqref{eq:step_23_k} and the second from Observation \ref{obs:4caterpillar}.
Insert these parameters in \eqref{eq:ramsey_2t+2} and \eqref{eq:ramsey_2t+3} to obtain respectively
\begin{align}
\tower_{2t+4}(\log C_0 + (L+2) \log C_1) &< r_{2L+5t+5}(5+2t;4+2t),r_{L+5t+3}(6+2t;4+2t),\label{eq:even_r} \\
\tower_{2t+5}(\log C_0 + (L+2) \log C_1) &< r_{2L+5t+7}(6+2t;5+2t),r_{L+5t+5}(7+2t;5+2t).\label{eq:odd_r}  
\end{align}
Any even $r\geq 4$ is of the form $r=2t+4$ (so $\lfloor 5 r /2\rfloor=5t+10$) and thus \eqref{eq:even_r} yields:
\begin{align*}
    \tower_{r}(\log C_0 + (L+2) \log C_1) &< r_{2L+\lfloor 5 r /2\rfloor-5}(r+1;r) \leq r_{(2L+1)+\lfloor 5 r /2\rfloor-5}(r+1;r),\\
    \tower_{r}(\log C_0 + (L+2) \log C_1) &< r_{L+\lfloor 5 r /2\rfloor-7}(r+2;r),
\end{align*}
where the rightmost inequality in the first line follows from Observation \ref{obs:monotony}. The same equations are also obtained for odd $r\geq 5$, this time by setting $r=2t+5$ in $\eqref{eq:odd_r}$. This proves \eqref{eq:bound_r+1} for all values $k\geq \lfloor 5 r /2\rfloor -5$ (respectively \eqref{eq:bound_r+2} and $ \lfloor 5 r /2\rfloor -7$).
\end{proof}

To the best of our knowledge, the current best lower bound for $r_k(5;3)$ (for large $k$) is $\tower_3(k + O(1))$ \cite{BBH18}. Compare with our $r_k(5;3)> \tower_3(1.678k +O(1))$.

\section{Subset-colouring of hypergraphs}
\label{sec:subset_colouring}

Recall that, given a hypergraph $H=(V,E)$ and $r \in \Nat$, the integer $k=k(H;r)$ is the least number of colours for which there is a colouring $\subsets V r \to \iinterval k$ without any monochromatic hyperedge $e \in E$ of size $\geq r+1$, and $k(n,r)=\max \{k(H;r)\colon \chi(H)=n\}$. As noted by Sarkaria~\cite{Sarkaria83}, the problem of determining $k(n,r)$ is closely related to the \enquote{classical} Ramsey problem discussed in previous sections. Indeed, for $n \in \Nat$, consider the complete $(r+1)$-uniform hypergraph on $rn$ vertices $K_{rn}^{(r+1)}$. Clearly, $\chi (K_{rn}^{(r+1)})=n$, and we have 
\begin{observation}\label{obs:EquivB}
For any $n,r\in \Nat$,
\begin{equation}\label{EquivB}
   k(K_{rn}^{(r+1)};r)  =  \min \{k \colon r_k(r+1;r) > rn \}.
\end{equation}
\end{observation}

\begin{proof}
By definition, $k(K_{rn}^{(r+1)};r)$ is the least $k$ for which there is a colouring $\subsets {\iinterval{rn}} r \to \iinterval{k}$ without any monochromatic $(r+1)$-subset, i.e.\@ such that $rn \nerarrow (r+1)_k^r$, or equivalently such that $rn < r_k(r+1;r)$.
\end{proof}

Our main result in this section is that the values $k(n,r)$ and $k(K_{rn}^{(r+1)};r)$ are very close to each other:
\begin{theorem}\label{thm:equivalence} 
For any positive integers $n$ and $r$, 
\begin{equation}\label{Eq:compare1}
k(K_{rn}^{(r+1)};r) \leq k(n,r) \leq k(K_{rn}^{(r+1)};r)+5.
\end{equation}
\end{theorem}

Theorem \ref{thm:equivalence} will follow from Theorem \ref{thm:subsets} below. From the above Observation and Theorem we conclude that the problem of finding $k(n,r)$ is essentially equivalent to the problem of finding the Ramsey number $r_{k'}(r+1;r)$ for an appropriate value of $k'$. 

\begin{theorem}
\label{thm:subsets}
Let $r,k \geq 2$ and let $n$ be such that $n \nerarrow (r+1)_k^r$. If the hypergraph $(V,E)$ admits a vertex-colouring $V\to\iinterval n$ under which no hyperedge of size $\geq r+1$ is monochromatic, then there is an $r$-subset colouring $\subsets{V}{r}\to\iinterval{k+f(r)}$ such that no hyperedge of size $\geq r+1$ is monochromatic. Here
\[
f(r)=\begin{cases}
    1 &\text{if } r=2, \\
    3 &\text{if } r=3, \\
    4 &\text{if } r\geq 4 \text{ and } r+1 \text{ is prime,}\\
    5 &\text{otherwise.}
\end{cases}
\]
\end{theorem}

\begin{remark}
\label{rem:follow}
In Theorem~\ref{thm:subsets} we require only hyperedges of size at least $(r+1)$ not to be monochromatic. A stronger but more natural condition is for $(V,E)$ to have \emph{vertex chromatic number} at most $n$.
\end{remark}

\begin{proof}[Proof of Theorem~\ref{thm:equivalence}]
\label{rem:subsets}
Let $k=k(K_{rn}^{(r+1)};r)$, so that $rn\nerarrow (r+1)_{k}^r$ and, by monotony, $n \nerarrow (r+1)_{k}^r$. Theorem~\ref{thm:subsets} then states that any hypergraph with vertex chromatic number $n$ has an $r$-subset colouring with $k+5$ colours and no monochromatic hyperedge of size $r+1$, i.e.\@ $k(n,r) \leq k+5$. The left inequality of~\eqref{Eq:compare1} is by definition.
\end{proof}

\begin{remark}
The two-sided relation with Ramsey numbers in Theorem \ref{thm:equivalence} holds only for $k(n,r)$. For a specific hypergraph $H$, $k(H;r)$ may be much smaller than $k(K_{r\chi(H)}^{(r+1)};r)$. For example, fix $r$ and let $H$ be the complete $r$-uniform hypergraph on $\ell$ vertices $K_\ell^{(r)}$. We have $k(H;r)=1$ but $r \chi(H)> \ell$ so $k(K_{r\chi(H)}^{(r+1)};r) \geq \min\{ k \colon \ell \nerarrow (r+1)^r_k\} \to \infty$ as $\ell$ increases.
\end{remark}

\begin{proof}[Proof of Theorem~\ref{thm:subsets}.]
Let $C_1:V \to \iinterval n$ be the given proper colouring of $V$.
Let $C_2: \subsets{\iinterval n}{r}  \to \iinterval{k}$ be an $r$-subset colouring without monochromatic set in $\subsets {\iinterval n} {r+1}$. 

In order to describe the colouring, we need some more definitions. 
An \emph{($m$-part) partition} of $r$ is a non-increasing sequence of positive integers $(r_0, \dots , r_{m-1})\in (\Nat\setminus\{0\})^m$ with $\sum_{i\in\iinterval m} r_i =r$. When $m>1$, this partition is \emph{balanced} if $r_0=\dots=r_{m-1}$, and \emph{almost balanced} if $r_0=\dots=r_{m-2}$ and $r_{m-1}=r_0-1$.

To every partition $R=(r_0, \dots , r_{m-1})$ we associate an ordered pair in ${\iinterval 2}^2$:
\[ C_0(R) =  \left(
\begin{cases}0 &\text{if $m+r_{m-1}$ is even,}\\ 1 &\text{otherwise}\end{cases},   
\begin{cases}
0  &\text{if } r_{m-1}=1,\\
1  &\text{otherwise} 
\end{cases} \right). \]

In addition, given $C_0(R)= (a,b)$ we define $\overline{C_0}(R)= k + 2a + b\in \{k,k+1,k+2,k+3\}$.

Finally, the \emph{histogram} of a sequence $(x_i)_{i\in[r]}$ is the unique partition $(r_0, \dots, r_{m-1})$ of $r$ such that $(x_i)_{i\in[ r]}$ contains exactly $m$ distinct values $y_0, \dots y_{m-1}$, and $y_i$ appears exactly $r_i$ times. For example, the histogram of the $6$-sequence $(9,7,7,7,15,15)$ is $(3,2,1)$, a partition of $6$.

We are ready to define the $r$-subset colouring $C\colon \subsets{V}{r}\to\iinterval{k+5}$. Given $\{x_1, \dots, x_r\}\in\subsets{V}{r}$, write $c_i = C_1(x_i)$ for every $i\in [r]$. Let $S=(r_0,\dots,r_{m-1})$ be the histogram of the sequence $(c_i)_{i\in[r]}$. We let
\[
C(\{ x_1, \dots, x_r\})=
\begin{cases}
C_2(\{c_i \colon i\in \iinterval r\})\in\iinterval k  &\text{if $S=(1,1,\dots,1)$},\\
k+4  &\begin{aligned}
    &\text{if $S$ is almost balanced}\\ &\text{and }\min \{c_i \colon i \in \iinterval r\} \text{ appears} \\
    &\text{exactly } r_{m-1} \text{ times,}
    \end{aligned}\\
 \overline{C}_0(S)\in \{k,k+1,k+2,k+3\}  &\text{otherwise.}
\end{cases}
\]  

For example, if $r=5, n>15,$ and some 5-subset $\{x_1, \dots,x_5\}$ of $V$ gets in $C_1$ the colours $(2,3,3,3,15)$, then $S=(3,1,1)$, and so $C( \{x_1,\dots,x_5\}  )  = k +2 \cdot 0 +0 =k$, since $m+r_{m-1}=3+1 $ is even, and $r_{m-1}=r_2=1$. 
 
\paragraph{Validity of the colouring.}
Without loss of generality it suffices to prove that any hyperedge of size $r+1$, that is $e=\{x_1, \dots, x_{r+1}\}\in E$, contains two $r$-subsets receiving distinct colours in the colouring $C$. Let $R$ be the histogram of the sequence $(C_1(x_i))_{i \in [r+1]}$.

Let us first eliminate one of the possible cases in the colouring. 

\begin{observation}
\label{obs:balanced}
The hyperedge $e$ is not monochromatic in the colour $k+4$.
\end{observation}

\begin{proof}[Proof of Observation \ref{obs:balanced}]
If $R$ is balanced, then clearly $e$ contains an $r$-subset that is not coloured $k+4$. If $R$ is not balanced, assume to the contrary that all $r$-subsets of $e$ are coloured $k+4$. In particular their histograms are all almost balanced.

Let $(\underbrace{a,\dots,a,}_{\text{$m+1$ places}} a-1)$ be the almost balanced histogram of some $r$-subset of $e$ with $m\geq 0$ and $a\geq 2$. Then $R$ is either (I) $(a+1,\underbrace{a,\dots,a,}_{\text{$m$ places}}a-1)$ or (II) $(\underbrace{a,\dots,a,}_{\text{$m+1$ places}} a-1,1)$. 

In case (I), we have $m>0$ as otherwise depending on whether $a>2$ some $r$-subset has the histogram $(3)$ or $(a+1,a-2)$, and neither is almost balanced. Then we must have $m=1$, $a=2$ and $R=(3,2,1)$, but in this case $r=5$ and $e$ contains a 5-subset with histogram $(3,1,1)$, which is not almost balanced either.

In case (II), as by assumption $(\underbrace{a,\dots,a,}_{\text{$m$ places}} a-1,a-1,1)$ is almost balanced,we have $m=0$ and $a=3$. Hence $R=(3,2,1)$, still a contradiction.
\end{proof}

It follows from Observation \ref{obs:balanced} that whenever $e$ contains an $r$-subset coloured with $k+4$, $e$ is not monochromatic in the colouring $C$. Hence, in the analysis below (specifically, in cases 3-5), no $r$-subset of $e$ is ever monochromatic in colour $k+4$.

As $e$ is not monochromatic under $C_1$ we have $m\geq 2$. Five cases cover all remaining possible values of $R=(r_0,\dots,r_{m-1})$. Let $\delta_{ij}$ be $1$ if $i=j$ and $0$ otherwise.

\begin{itemize}
\item  \textbf{Case 1:} $r_0=r_{m-1}=1$, i.e.\@ the $C_1(x_i)$ are all distinct. Then, by the definition of $C_2$, $\{C_1(x_i) \colon i \in [r+1]\}$ contains two $r$-subsets that admit two distinct colours in $C_2$, and therefore also in $C$.

\item \textbf{Case 2:} $r_0=r_{m-1} > 1$. In this case, some $r$-subset of $e$ is coloured in $C$ with $k+4$, and by Observation \ref{obs:balanced} we are done. 

\item \textbf{Case 3:} $r_0 > r_{m-1}>1$. In this case, some $r$-subset of $e$ has the histogram $R_1=(r_0, \dots, r_{m-2},r_{m-1}-1)$, while another $r$-subset of $e$ has the histogram $R_2=(r_i -\delta_{ij})_{i\in\iinterval{m}}$ for some $j<m-1$. Since $C_0(R_1)$ and $C_0(R_2)$ differ in the first coordinate, $\overline{C}_0(R_1) \neq \overline{C}_0(R_2)$ hence these two $r$-subsets admit two distinct colours in $C$.

\item \textbf{Case 4:} $r_{m-2} > r_{m-1}=1$. In this case, some $r$-subset of $e$ has the histogram $R_1=(r_0, \dots, r_{m-2})$, while another $r$-subset of $e$ has the histogram $R_2=(r_0, r_1, \dots,r_{m-2}-1,1)$. Since $C_0(R_1)$ and $C_0(R_2)$ differ in the second coordinate, $\overline{C}_0(R_1) \neq \overline{C}_0(R_2)$ hence these two $r$-subsets admit two distinct colours in $C$.

\item \textbf{Case 5:} $r_0 >r_{m-2}=r_{m-1}=1$. In this case, some $r$-subset of $e$ has the histogram $R_1=(r_0, \dots, r_{m-2})$, while another $r$-subset of $e$ has the histogram  $R_2=(r_i -\delta_{ij})_{i\in\iinterval{m}}$ for some $j< m-2$. Since $C_0(R_1)$ and $C_0(R_2)$ differ in the first coordinate (note that $r_{m-1}=r_{m-2}=1$), $\overline{C}_0(R_1) \neq \overline{C}_0(R_2)$ hence these two $r$-subsets admit two distinct colours in $C$.
\end{itemize}

Finally, note that for specific values of $r$, some colours are left unused:
For $r=2$, the colours used are in $\iinterval k\cup \{k+2\}$. For $r=3$, they are in $\iinterval k \cup \{k+1,k+2,k+4\}$. Finally, if $r+1$ is prime then $R$ cannot be balanced without being $(1,1,\dots,1)$. In this case we can define $C$ without using the colour $k+4$ (whose role in the proof is to handle balanced partitions of $r+1$). This justifies having the number of colours $f(r)$ as in the statement of the theorem.
\end{proof}

For the next result, let $\operatorname{Log} x = \max\{1,\log x\}$.
\begin{corollary}
\label{cor:UBhypergraphs}
For $n\geq r\geq 2$, we have
    \begin{equation}
        \frac{1}{3} \cdot \frac{ \operatorname{Log}^{(r-1)} (rn)}{ \operatorname{Log}^{(r)} (rn)}   +o(1) \leq k(n,r)\leq \frac{2}{\alpha} \operatorname{Log}^{(r-1)}(rn) +\frac{5r}{2} + O(1).
    \end{equation}
\end{corollary}

\begin{proof}
Combining Observation \ref{obs:EquivB} and Theorem \ref{thm:equivalence}, we have 
\[\min \{k \colon r_k(r+1;r) > rn \} \leq k(n,r) \leq  \min \{k \colon r_k(r+1;r) > rn \} + 5.\]

For $rn$ large enough, let $k_0=\lceil\frac{2}{\alpha}(\log^{(r-1)}(rn)-\beta)+\frac{5r}{2} \rceil \geq  \lceil 5 r/2\rceil $. By Corollary \ref{cor:Ramsey_bounds}, we have \[	r_{k_0}(r+1;r)> \tower_r \left(\frac \alpha 2  \cdot \left(k_0-\frac{5r}{2}\right) + \beta\right)\geq rn,\]
thus $k(n,r)\leq k_0 +5$. 

Similarly, also for large $rn$, let $k_1 =\lfloor \frac{1}{3} \cdot \frac{ \log^{(r-1)} (rn)   } {  \log^{(r)} (rn)   }\rfloor$. By Theorem \ref{thm:ER}, we have $r_{k_1}(r+1;r)< \tower_r(3 k_1 \log k_1 ) \leq  rn$. By the monotony of $r_k(r+1;r)$ in $k$ (Observation \ref{obs:monotony}), we conclude $k(n,r)>k_1$.
\end{proof}

\printbibliography

\appendix

\section{Infinitary Ramsey combinatorics}

In this section we sketch the modifications needed to extend our stepping-up lemma to infinite cardinals. 
\label{app:infinite}

We work in Zermelo--Fraenkel set theory with choice (ZFC), and we use von Neumann ordinals (so for integers we have $n=\iinterval n$, more generally any ordinal $\alpha = \{\beta\colon \beta < \alpha\}$, and the cardinal of any set $S$ is $|S|=\min \{ \alpha \text{ ordinal} \colon \alpha \simeq S\}$). As usual we also let $\omega = |\Nat| = \Nat$. We write ${+}$ and ${\ordinalplus}$ for cardinal and ordinal addition respectively.

Let $r\geq 2$ be an integer, let $\tau>0$ be an ordinal, $\kappa>0$  and $\lambda_{\xi}\geq r+1$ (for every $\xi < \tau$) be cardinals. Extend the relation $\nerarrow$ previously defined for integers by calling $X\subset \kappa$ is monochromatic in colour $\xi<\tau$ under $f\colon \subsets{\kappa}{r}\to \tau$ if $\image f \subsets X r \subset\{\xi\}$ and writing
\[ \kappa \nerarrow (\lambda_{\xi})_{\xi< \tau}^{r} \]
if for some $f\colon \subsets{\kappa}{r}\to \tau$ there does not exist $\xi<\tau$ and a subset $X\subset \kappa$ monochromatic in $\xi$ such that $|X|=\lambda_{\xi}$.

Assuming the above relation holds and $r\geq 3$, \textcite[chap.\@ 24]{EHMR84} prove
\[
 2^{\kappa} \nerarrow ((\lambda_{\xi} +1)_{\xi< \tau}, (r+2)_{2^{r} + 2^{r-1} -4})^{r+1},
\]
which is shorthand for $2^{\kappa} \nerarrow (\lambda'_{\xi} +1)^{r+1}_{\xi< \tau \ordinalplus (2^{r} + 2^{r-1} -4)}$ where $\lambda'_\xi$ is $\lambda_\xi$ for $\xi<\tau$ and otherwise $r+1$. We can strengthen this result somewhat:
\begin{theorem}\label{thm:infinitary}
Let $r\geq 3$ be an integer, let $\kappa>0$ be a cardinal, $\tau>0$ an ordinal and $\lambda_{\xi}\geq r+1$ ($\xi < \tau$) cardinals.
\begin{gather*}
\text{If }  \kappa \nerarrow (\lambda_{\xi})_{\xi< \tau}^{r}
\text{ then } 
   2^{\kappa} \nerarrow ((\lambda_{\xi} +1)_{\xi< \tau}, \underbrace{r+2,\dots,r+2}_{\substack{\eta(r) \text{ terms,}\\\text{each $r+2$}}})^{r+1},
\end{gather*}
with the integer $\eta(r)\leq 3$  defined above in \eqref{eq:additional_colours}.
\end{theorem}

To the best of our knowledge, this improves upon known results in two cases: when $\kappa$ and all $\lambda_\xi$ are finite (corresponding to our previous finite results), and when $r=3$, $\kappa$ is infinite, $\lambda_0$ is infinite and singular and all other $\lambda_\xi$ are finite. In all other cases, the book by \citeauthor{EHMR84} contains stronger statements.

\begin{proof}[Proof sketch]
The hypothesis means that there exists $f_r\colon \subsets \kappa r \to \tau$ for which none of the colour classes has a monochromatic set of cardinality $\lambda_{\xi}$. Due to the bijection $2^\kappa \simeq |\powerset \kappa|$ it suffices to exhibit $f_{r+1} \colon \subsets {\powerset\kappa} {r+1} \to \tau \ordinalplus \eta(r)$ for which none of the first $\tau$ colour classes has a monochromatic set of cardinality $\lambda_{\xi}+1$, and none of the subsequent ones has a monochromatic set of cardinality $r+2$.

To do so define the first discrepancy between two sets of ordinals $x$ and $y$ with $x\neq y$ as
\begin{align*}
s(x,y) = \min ((x \cup y) \setminus (x\cap y)),
\end{align*}

which is the well-defined minimum of a non-empty set of ordinals. From there if $\sigma \subset \powerset \kappa$ is a set of subsets of $\kappa$ and $|\sigma|\geq 2$ we define its first splitting index as $s(\sigma)=\min (\image s \subsets{\sigma}{2})$, a well-defined ordinal in $\kappa$.

This gives \emph{essentially} the same partition as before with $\sigma = \splitting {\sigma_0}{\sigma_1}$ where $\sigma_0 = \{S\in \sigma\colon s(\sigma) \notin S\}$ and $\sigma_1 = \sigma \setminus \sigma_0$. It admittedly differs from the previous finitary construction in a minor way, as we now compare \emph{least} significant bits, but when restricted to finite sets it is functionally equivalent. This small modification is necessary; using maxima in the definitions above would make these quantities ill-defined.

A caterpillar is a set of any size without a four-subset of the form $\splitting L R$ with $|L|=|R|=2$. If $\sigma=\splitting{\sigma_0}{\sigma_1}$ is a caterpillar then at least one of $\sigma_0$ and $\sigma_1$ is a singleton. 

For any caterpillar $T\subset \powerset \kappa$ let $\delta(T)=\{s(x,y)\colon \{x,y\}\in \subsets T 2\}= \image s \subsets T 2$.

We (still) have $\subsets{\delta(T)}{r}= \image{\delta}{\subsets{T}{r+1}}$, extending Observation \ref{obs:subset_and_delta}. This infinite case is not harder to prove since
\[\subsets{\delta(T)}{r} = \bigcup_{\substack{Y\subset T\\ |Y|< \omega}} \subsets{\delta(Y)}{r} =  \bigcup_{\substack{Y\subset T\\ |Y|< \omega}} \image{\delta}{\subsets{Y}{r+1}} =\image\delta{\Big(\bigcup_{\substack{Y\subset T\\ |Y|< \omega}} {\subsets{Y}{r+1}}\Big)}= \image{\delta}{\subsets{T}{r+1}}.\] 

Similarly, the equality $|T|=|\delta(T)|+1$ still holds for every caterpillar $T$ by transfinite induction.

We re-use the same colouring scheme. Let $S\in \subsets{\powerset\kappa} {r+1}$. If $S\in\caterpillars$ then $\delta(S)\in \subsets{\kappa}{r}$, and we let $f_{r+1}(S)=f_{r}(\delta(S)) < \tau$. If $S\notin\caterpillars$, then $S$, being finite, has a type, so we simply colour it as before.

As before, let us check the validity of this colouring only for odd $r$.

Let $X\subset \powerset\kappa$ be monochromatic in colour $\xi$ under $f_{r+1}$. We will show that $|X| \leq \lambda_\xi$ (if $\xi< \tau$) or $|X|\leq r+1$ (for $\xi\in\{\tau,\tau+1\}$). As $r+1\leq \lambda_\xi$ we may already assume $|X|> r+1$.

First consider the case when $X\in \caterpillars$. Then $\subsets X {r+1} \subset \caterpillars$, and thus $\{\xi\} = \image{f_{r+1}}{\subsets X {r+1}} =\image{f_{r}}{(\image\delta {\subsets X {r+1}})} < \tau$, so we have to show $|X|\leq \lambda_\xi$.

The set $\delta(X)\subset \kappa$ is monochromatic in colour $\xi$ under $f_{r}$, as
\[ \image {f_{r}}{\subsets{\delta(X)}{r}} = \image {f_{r}}{(\image \delta {\subsets{X}{r+1}})}  = \{ \xi\}.
\]

This fact implies that $|X| = |\delta(X)| +1 < \lambda_\xi + 1$.

The other case to consider is when $X\notin \caterpillars$. We show that then $|X|\leq r+1$ regardless of $i$. (Recall that $r+1\leq s_i$ always.) Suppose to the contrary that $|X| \geq r+2$, and we can assume $|X|=r+2$. Now everything is finite, and works as before.
\end{proof}

\section{Lower bounds on \texorpdfstring{$r_k(3;2)$}{r(3;2)}}

Here is the justification for the lower bounds on $r_k(3;2)$ used in the proof of Corollary \ref{cor:Ramsey_bounds}.
\label{app:graph_bounds}

\begin{theorem}
    There is an absolute constant $C>0$ such that for all $k\geq 2$ we have 
    \[
    \lceil C \cdot 1073^{k/6} \rceil \nerarrow (3)_k^2.
    \]
\end{theorem}
\begin{proof}
    A set of integers $S$ is sum-free if $(S+S)\cap S$ is empty. The Schur number $s_k$ is the largest number for which $\{1,2,\dots,s_k\}$ has a partition into $k$ sum-free subsets. By Schur's Ramsey theorem, $s_k \nerarrow (3)_k^2$ (see e.g.\@ \cite[chap.\@ 4.5]{GRS90}). For any fixed positive integer $l$ and $k\geq l$ we have \cite{CG83}
    \[ s_k \geq C_l \cdot (2 s_l + 1)^{k/l}.\]
    Finally, $s_6 \geq 536$ \cite{FS00}.
\end{proof}
\end{document}